\documentclass{amsart}
\usepackage{amsmath,amsfonts,amssymb}
\usepackage{ifthen}
\usepackage{longtable}
\usepackage{array}
\usepackage{url}

\newcommand{\Z}{\mathbb{Z}}
\newcommand{\Q}{\mathbb{Q}}
\newcommand{\R}{\mathbb{R}}

\newcommand{\N}{\mathbb{N}}

\newcommand{\LL}{\mathcal{L}}

\newtheorem*{theorem*}{Theorem}
\newtheorem{theorem}{Theorem}
\newtheorem{lemma}{Lemma}

\newtheorem{proposition}{Proposition}

\newtheorem{problem}{Problem}

\theoremstyle{remark}
\newtheorem{remark}{Remark}

\begin{document}
\title[$S$-Diophantine quadruples]{Finding all $S$-Diophantine quadruples for a fixed set of primes $S$}

\author{Volker Ziegler}
\address{V. Ziegler,
University of Salzburg,
Hellbrunnerstrasse 34/I,
A-5020 Salzburg, Austria}
\email{volker.ziegler\char'100sbg.ac.at}
\thanks{The author was supported by the Austrian Science Fund(FWF) under the project~I4406.}

\begin{abstract}
Given a finite set of primes $S$ and a $m$-tuple $(a_1,\dots,a_m)$ of positive, distinct integers we call the $m$-tuple $S$-Diophantine, if for each $1\leq i < j\leq m$ the quantity $a_ia_j+1$ has prime divisors coming only from the set $S$. For a given set $S$ we give a practical algorithm to find all $S$-Diophantine quadruples, provided that $|S|=3$.   
\end{abstract}

\maketitle

\section{Introduction}

Given an irreducible polynomial $f\in \Z[X,Y]$ and a set $A$ of positive integers we consider the product
$$ \Pi=\prod_{\substack{a,b\in A \\ a\neq b}} f(a,b).$$
It is an interesting question what is the largest prime divisor $P(\Pi)$ of $\Pi$ or alternatively what is the number of prime divisors $\omega(\Pi)$ of $\Pi$ in terms of the size of $|A|$. In case of $f(x,y)=x+y$ this question was intensively studied by several authors starting with Erd\H{o}s and Tur\'an \cite{Erdos:1934} who established that in this case $\omega(\Pi)\gg \log |A|$. It was shown by Erd\H{o}s, Stewart and Tijdeman \cite{Erdos:1988} that this is essentially best possible.

In the case of $f(x,y)=xy+1$ the investigation was started by S\'ark\"ozy \cite{Sarkozy:1992}. It was shown by Gy\H{o}ry et.al. that $\omega(\Pi)\gg \log |A|$. However it seems that this lower bound is far from being best possible and the situation is not so clear as in the case $f(x,y)=x+y$.

One recent approach to this problem was picked up by Szalay and Ziegler \cite{Szalay:2013a} who studied so-called $S$-Diophantine $m$-tuples. Let $S$ be a finite set of primes and  $(a_1,\dots,a_m)$ a $m$-tuple of positive, distinct integers with $a_1<\dots <a_m$. Then we call $(a_1,\dots,a_m)$ $S$-Diophantine, if for each $1\leq i < j\leq m$ the quantity $a_ia_j+1$ has prime divisors coming only from the set $S$. In other words if $A=\{a_1,\dots,a_m\}$ and $\Pi$ has only prime divisors coming form $S$, then $(a_1,\dots,a_m)$ is called a $S$-Diophantine $m$-tuple. Szalay and Ziegler~\cite{Szalay:2013a} conjectured that
if $|S|=2$, then no $S$-Diophantine quadruple exists. This conjecture has been confirmed for several special cases by Luca, Szalay and Ziegler~\cite{Szalay:2013a,Szalay:2013b,Szalay:2015,Ziegler:2019,Luca:2019}. Even a rather efficient algorithm has been described by Szalay and Ziegler \cite{Szalay:2015} that finds for a given set $S=\{p,q\}$ of two primes all $S$-Diophantine quadruples, if there exist any. However, the algorithm is strictly limited to the case that $|S|=2$ and does not work (nor any slightly modified version of the algorithm) for $|S|\geq 3$.

In the present paper we establish an algorithm that finds for any set $S=\{p,q,r\}$ of three primes all $S$-Diophantine quadruples. Unfortunately the algorithm is not so efficient as its counter part for $|S|=2$. However, when implemented in Sage \cite{sage} it takes on a usual PC
(Intel i7-8700) -- using a single core -- about six and a half hours to verify: 

\begin{theorem}\label{th:main}
 There is no $\{2,3,5\}$-Diophantine quadruple.
\end{theorem}

However the case $S=\{2,3,5\}$ is in some sense the slowest instance for small primes $p,q$ and $r$. Thus within four days of runtime we could establish the following result:

\begin{theorem}\label{th:search}
 Let $S=\{p,q,r\}$ with $2\leq p<q<r\leq 100$. Then all $S$-Diophantine quadruples are listed in Table \ref{tab:list}.
\end{theorem}

\begin{table}[ht]
    
    \begin{tabular}{|c|c|} 
    \hline $S$ & Quadruples \\\hline\hline
    $\{2,3,11\}$ & $(1,3,5,7)$ \\\hline
     $\{2, 3, 29\}$ & $(1, 5, 7, 23)$ \\\hline
       $\{2, 11, 37\}$ & $(1, 3, 7, 21)$ \\\hline
    \end{tabular}
    \caption{$S$-Diophantine quadruples.}
    \label{tab:list}
\end{table}

Thus Theorem \ref{th:search} raises the following problem:

\begin{problem}
 Are the $S$-Diophantine quadruples listed in Table \ref{tab:list} all $S$-Diophantine quadruples with $|S|=3$?
\end{problem}

Although we present the method only in the case that $|S|=3$ obvious modifications to the algorithm would provide an algorithm that would work for any set $S$ of primes with $|S|\geq 3$. However, already in the case that $|S|=4$ the algorithm would take instead of several hours several months of computation time. Thus a systematic search for $S$-Diophantine quadruples in the case $|S|=4$ seems to be not feasible.  Therefore we refer from discussing the case $|S|>3$ in this paper.

In the next section we remind some of the auxiliary results we need to deal with $S$-Diophantine quadruples. In Section \ref{sec:bound} we establish several results that will allow us to find small upper bounds for the relevant unknowns and will also allow us to apply LLL-reduction to reduce these bounds. The LLL-reduction will be discussed in Section \ref{sec:LLL}. In Section \ref{sec:reduction} we will give details how to apply the LLL-reduction method to find upper bounds small enough to  enumerate all possible $S$-Diophantine quadruples for a given set $S=\{p,q,r\}$ of three primes.  

\section{Auxiliary results}

Assume that $(a,b,c,d)$ is a $S$-Diophantine quadruple, with $S=\{p,q,r\}$. In particular, we assume that $a<b<c<d$. Moreover, we will assume that $p<q<r$. Let us write
\begin{align*}
 ab+1=s_1 &=p^{\alpha_1}q^{\beta_1}r^{\gamma_1} &
 bc+1=s_4 &=p^{\alpha_4}q^{\beta_4}r^{\gamma_4} \\
 ac+1=s_2 &=p^{\alpha_2}q^{\beta_2}r^{\gamma_2}&
 bd+1=s_5 &=p^{\alpha_5}q^{\beta_5}r^{\gamma_5} \\
 ad+1=s_3 &=p^{\alpha_3}q^{\beta_3}r^{\gamma_3}&
 cd+1=s_6 &=p^{\alpha_6}q^{\beta_6}r^{\gamma_6}.
\end{align*}
The almost trivial observation that
$$ab \cdot cd=(s_1-1)(s_6-1)=(s_2-1)(s_5-1)=ac\cdot bd$$
 yields the (non-linear) $S$-unit equation
$$s_1s_6-s_1-s_6=s_2s_5-s_2-s_5.$$
Similar computations lead us to the following system of $S$-unit equations
\begin{equation}\label{eq:Sunit-sys-gen}
 \begin{split}
  s_1s_6-s_1-s_6&=s_2s_5-s_2-s_5,\\
  s_1s_6-s_1-s_6&=s_3s_4-s_3-s_4,\\
  s_2s_5-s_2-s_5&=s_3s_4-s_3-s_4.
 \end{split}
\end{equation}
Let us call two indices $i,j\in\{1,\dots,6\}$ complementary, if $i+j=7$. Note that if $i$ and $j$ are complementary, then we have $(s_i-1)(s_j-1)=abcd$. With these notations at hand the following lemmas can be proved (e.g. see \cite{Szalay:2013a}) by elementary means.

\begin{lemma}\label{lem:min_alpha}
 The smallest two members of each of the sets 
 $$\{\alpha_1,\alpha_2,\alpha_5,\alpha_6\},\quad \{\alpha_1,\alpha_3,\alpha_4,\alpha_6\},\quad \{\alpha_2,\alpha_3,\alpha_4,\alpha_5\}$$
 are equal. Similar statements also hold for the $\beta$'s and $\gamma$'s
\end{lemma}

\begin{proof}
 This follows by $p$-adic considerations of the system of equations \eqref{eq:Sunit-sys-gen} (cf. \cite[Proposition 1]{Szalay:2013a}).
\end{proof}

\begin{lemma}\label{lem:abcd-div}
 Assume that $(a,b,c,d)$ is a $S$-Diophantine quadruple. Then we have
\begin{align*}
a&|\gcd\left(\frac{s_2-s_1}{\gcd(s_2,s_1)},\frac{s_3-s_1}{\gcd(s_3,s_1)},\frac{
s_3-s_2}{\gcd(s_3,s_2)}\right),\\
b&|\gcd\left(\frac{s_4-s_1}{\gcd(s_4,s_1)},\frac{s_5-s_1}{\gcd(s_5,s_1)},\frac{
s_5-s_4}{\gcd(s_5,s_4)}\right),\\
c&|\gcd\left(\frac{s_4-s_2}{\gcd(s_4,s_2)},\frac{s_6-s_2}{\gcd(s_6,s_2)},\frac{
s_6-s_4}{\gcd(s_6,s_4)}\right),\\
d&|\gcd\left(\frac{s_5-s_3}{\gcd(s_5,s_3)},\frac{s_6-s_3}{\gcd(s_6,s_3)},\frac{
s_6-s_5}{\gcd(s_6,s_5)}\right).
\end{align*}
\end{lemma}

\begin{proof}
 A proof can be found in \cite[Lemma 3]{Szalay:2013a}.
\end{proof}

\begin{lemma}\label{lem:quad-res}
 Let $(a,b,c)$ be a $S$-Diophantine triple with 
 \begin{align*}
  ab+1&=s_1=p^{\alpha_1}q^{\beta_1}r^{\gamma_1}\\
  ac+1&=s_2=p^{\alpha_2}q^{\beta_2}r^{\gamma_2}\\
  bc+1&=s_4=p^{\alpha_4}q^{\beta_4}r^{\gamma_4}.
  \end{align*}
  Then we have:
  \begin{itemize}
  \item If $p=2$, then $\min\{\alpha_1,\alpha_2,\alpha_4\}\leq 1$.
  \item If $p\equiv 3\mod 4$, then $\min\{\alpha_1,\alpha_2,\alpha_4\}=0$.
  \item If $q\equiv 3\mod 4$, then $\min\{\beta_1,\beta_2,\beta_4\}=0$.
  \item If $r\equiv 3\mod 4$, then $\min\{\gamma_1,\gamma_2,\gamma_4\}=0$.
  \end{itemize}
\end{lemma}

\begin{proof}
 A proof can be found in \cite[Lemma 2.2]{Szalay:2013b} except for the case that $p=2$. However, in the case that $p=2$ we slightly change the proof. That is we assume that $\min\{\alpha_1,\alpha_2,\alpha_4\}\geq 2$ and obtain
 $$ a^2b^2c^2=(s_1-1)(s_2-1)(s_3-1)\equiv -1\equiv 3 \mod 4$$
 which is an obvious contradiction.
\end{proof}

In order to obtain a first upper bound in the next section we will apply results on lower bounds for linear forms in (complex and $p$-adic) logarithms. Therefore let $\eta\neq 0$ be an algebraic number of degree $\delta$ and let
$$a(X-\eta_1)\cdots (X-\eta_\delta) \in \Z[X]$$
be the minimal polynomial of $\eta$. Then the absolute logarithmic Weil height is defined by
$$h(\eta)=\frac 1\delta \left(\log |a|+\sum_{i=1}^\delta \max\{0,\log|\eta_i|\}\right).$$
We will mainly need the notation of height for rational numbers. Therefore let us remark that if $p/q \in \Q$ with $p,q$ integers such that $\gcd(p,q)=1$, then $h(p/q)=\max\{\log |p|,\log |q|\}$. 
With this basic notation we have the following result on lower bounds for linear forms in (complex) logarithms due to Matveev \cite{Matveev:2000}.

\begin{lemma}\label{lem:Matveev}
Denote by $\eta_1,\ldots,\eta_n$ algebraic numbers, nor
$0$ neither $1$, by $\log \eta_1,\ldots$, $\log \eta_n$
determinations of their logarithms, by $D$ the degree over $\Q$ of
the number field $K = \Q(\eta_1,\ldots,\eta_n)$, and by $b_1,\ldots,b_n$
rational integers. Furthermore let $\kappa=1$ if $K$ is real
and $\kappa=2$ otherwise. Choose
$$A_i\geq \max\{D h(\eta_i),|\log \alpha_i|\} \quad (1\leq i\leq n)$$
and
\[E=\max\{1,\max \{ |b_j| A_j /A_n: 1\leq j \leq n \}\}.\]
Assume that $b_n\neq 0$ and
$\log \eta_1,\ldots,\log \eta_n$ are linearly independent
over $\Z$. Then
\[\log |b_1\log \eta_1+\cdots+b_n \log \eta_n|\geq -C(n)C_0 W_0 D^2
\Omega,\]
with
\begin{gather*}
\Omega=A_1\cdots A_n, \\
C(n)=C(n,\kappa)= \frac {16}{n! \kappa} e^n (2n +1+2 \kappa)(n+2)
(4(n+1))^{n+1} \left( \frac 12 en\right)^{\kappa}, \\
C_0= \log\left(e^{4.4n+7}n^{5.5}D^2 \log(eD)\right), \quad W_0=\log(1.5eED
\log(eD)).
\end{gather*}
\end{lemma}

In our case of main interest where $|S|=3$ we can apply results on linear forms in two $p$-adic logarithms. For a prime $p$ we denote by $\Q_p$ the field of $p$-adic numbers with the standard $p$-adic valuation $\nu_p$.
As above let $\eta_1,\eta_2$ be algebraic numbers over $\Q$ and we regard them as elements of the field $K_p=\Q_p(\eta_1,\eta_2)$. In our case $\eta_1$ and $\eta_2$ will be rational integers, thus $K_p=\Q_p$. Similar as in the case of Matveev's theorem above we have to use a modified height. In particular, we write
$$
h'(\eta_i) \geq \max\left\{h(\eta_i),\log p \right\}, \ (i=1,2).
$$

With these notations at hand, let us state the result due to Bugeaud and Laurent~\cite[Corollary 1]{Bugeaud:1996}):

\begin{lemma}\label{lem:Bugeaud}
Let $b_1,b_2$ be positive integers and suppose that $\eta_1$ and $\eta_2$ are multiplicatively independent algebraic numbers such that $\nu_p(\eta_1)=\nu_p(\eta_2)=0$. Put
$$
E':=\frac{b_1}{h'(\eta_2)}+\frac{b_2}{h'(\eta_1)}.
$$
and
$$
\label{eq:B}
E:=\max\left\{\log E'+\log\log p+0.4,10,10\log p\right\}.
$$
Then we have
$$
\nu_p(\eta_1^{b_1}\eta_2^{b_2}-1)\leq \frac{24 p}{(\log p)^4}E^2h'(\eta_1)h'(\eta_2).
$$

\end{lemma}

The next lemma is an elementary, but rather useful result due to Peth{\H o} and de Weger \cite{Pethoe:1986}. For a proof of Lemma \ref{lem:pdw} we refer to \cite[Appendix B]{Smart:DiGL}.

\begin{lemma} \label{lem:pdw}
Let $u,v \geq 0, h \geq 1$ and $x \in \R$ be the largest solution of $x=u+v(\log{x})^h$. Then
$$
x<\max\{2^h(u^{1/h}+v^{1/h}\log(h^hv))^h, 2^h(u^{1/h}+2e^2)^h\}.
$$
\end{lemma}

\section{A first upper bound}\label{sec:bound}

As already mentioned in the previous section we consider the case that $|S|=3$, i.e. we have $S=\{p,q,r\}$ with $P=\max\{p,q,r\}$ and write
\begin{align*}
 ab+1=s_1 &=p^{\alpha_1}q^{\beta_1}r^{\gamma_1} &
 bc+1=s_4 &=p^{\alpha_4}q^{\beta_4}r^{\gamma_4} \\
 ac+1=s_2 &=p^{\alpha_2}q^{\beta_2}r^{\gamma_2} &
 bd+1=s_5 &=p^{\alpha_5}q^{\beta_5}r^{\gamma_5} \\
 ad+1=s_3 &=p^{\alpha_3}q^{\beta_3}r^{\gamma_3} &
 cd+1=s_6 &=p^{\alpha_6}q^{\beta_6}r^{\gamma_6}.
\end{align*}
Let us write $A=\max_{i=1,\dots,6}\{\alpha_i\}, B=\max_{i=1,\dots,6}\{\beta_i\}$ and $C=\max_{i=1,\dots,6}\{\gamma_i\}$ and $M=\max\{A,B,C\}$.

The main result of this section is the following proposition:

\begin{proposition}\label{prop:first-upper-bound}
 If there exists an $S$-Diophantine quadruple, then 
 $$M<  3.62 \cdot 10^{23} (\log P)^5 \log( 3.62\cdot 10^{23} (\log P)^5)^2.$$
 In the case that $S=\{2,3,5\}$ we have the upper bound
 $$M<1.26\cdot 10^{28}$$
 and in the case that $S=\{p,q,r\}$ with $p,q,r<100$ we have the upper bound
 $$M<2.88 \cdot 10^{30}.$$
 
\end{proposition}

First, we prove an upper bound for $\alpha_i,\beta_i$ and $\gamma_i$ with $i=1,2,4$.

\begin{lemma}\label{lem:bound124}
 We have 
 $$\log s_1,\log s_2\leq 2.38 \cdot 10^{10} \log p \log q \log r \log (2M)$$
 and
 $$\log s_4\leq 4.76 \cdot 10^{10} \log p \log q \log r \log (2M).$$
 In particular, we have
 \begin{align*}
  \alpha_1,\alpha_2 & \leq 2.38 \cdot 10^{10} \log q \log r \log (2M),\\
  \beta_1,\beta_2 & \leq 2.38 \cdot 10^{10} \log p \log r \log (2M),\\
  \gamma_1,\gamma_2& \leq 2.38 \cdot 10^{10} \log p \log q \log (2M),
  \end{align*}
  and 
  \begin{align*}
  \alpha_4 & \leq 4.76 \cdot 10^{10} \log q \log r \log (2M),\\
  \beta_4 & \leq 4.76 \cdot 10^{10} \log p \log r \log (2M),\\
  \gamma_4 & \leq 4.76 \cdot 10^{10} \log p \log q \log (2M).
  \end{align*}
\end{lemma}

\begin{proof}
 Following Stewart and Tijdeman \cite{Stewart:1997} we consider the quantity
 \begin{equation}\label{eq:T1}
  T_1=p^{A'}q^{B'}r^{C'}=\frac{(ab+1)(cd+1)}{(ac+1)(bd+1)}=1+\frac{cd+ab-ac-bd}{(ac+1)(bd+1)}=1+\frac{\theta}{ab+1}
 \end{equation}
with $|\theta|< 1$, $A'=\alpha_1+\alpha_6-\alpha_2-\alpha_5$, $B'=\beta_1+\beta_6-\beta_2-\beta_5$ and $C'=\gamma_1+\gamma_6-\gamma_2-\gamma_5$. Note that we have $|A'|\leq 2A$, $|B'|\leq 2B$ and $|C'|\leq 2C$. Let us note that $\log(1+x)\leq 2x$ provided that $|x|\leq 1/2$. Since $ab+1\geq 2$ we obtain by taking logarithms
\begin{equation}\label{eq:Lamb1}
 |\Lambda_1|=|A' \log p+B'\log q+C'\log r|<\frac{2}{p^{\alpha_1}q^{\beta_1}r^{\gamma_1}}=\frac{2}{s_1}
\end{equation}
We apply Matveev's theorem (Lemma \ref{lem:Matveev}) with
\begin{gather*}
 D=1, \quad n=3, \quad \kappa=1, \quad b_1=A',\quad b_2=B',\quad b_3=C', \\
 \eta_1=p, \quad \eta_2=q, \quad \eta_3=r, \quad A_1=\log p, \quad A_2=\log q, \quad A_3=\log r.
\end{gather*}
Therefore we have $E\leq \max\{A',B',C'\}\leq 2M$ provided $M>2$ which we clearly may assume.
Thus we obtain 
$$\log s_1 -\log 2 < 2.375 \cdot 10^{10} \log p \log q \log r \log (2M),$$
which yields the upper bound for $\log s_1$. Since
$$\log s_1=\alpha_1\log p+\beta_1 \log q+\gamma_1\log r $$ 
this also yields the upper bounds for $\alpha_1,\beta_1$ and $\gamma_1$.

If we consider instead of $T_1$ the quantity
\begin{equation}\label{eq:T2}
  T_2=p^{A''}q^{B''}r^{C''}=\frac{(ac+1)(bd+1)}{(bc+1)(ad+1)}=1+\frac{bd+ac-bc-ad}{(bc+1)(ad+1)}=1+\frac{\theta}{ac+1}
 \end{equation}
with $|\theta|< 1$, $A''=\alpha_2+\alpha_5-\alpha_3-\alpha_4$, $B''=\beta_2+\beta_5-\beta_3-\beta_4$ and $C''=\gamma_2+\gamma_5-\gamma_3-\gamma_4$. We end up with the linear form
\begin{equation}\label{eq:Lamb2}
 |\Lambda_2|=|A'' \log p+B''\log q+C''\log r|<\frac{2}{p^{\alpha_2}q^{\beta_2}r^{\gamma_2}}
\end{equation}
and again an application of Matveev's result yields the same upper bound for $\log s_2$ and also the same upper bounds for $\alpha_2,\beta_2$ and $\gamma_2$.

Finally let us note that 
$$s_4= bc+1<(ab+1)(ac+1)=s_1s_2, $$
which yields after some easy computations the stated upper bounds for $\log s_4, \alpha_4,\beta_4$ and $\gamma_4$.
\end{proof}

Let us denote by $M_0$ and $M_4$ upper bounds for $\max\{\log s_1,\log s_2\}$ and $\log s_4$ respectively. Then the previous lemma provides upper bounds for $M_0$ and $M_4$.

\begin{lemma}\label{lem:bound-absolute-easy}
 We have $M\leq  3.79 \cdot 10^{12} M_0 (\log P)^2 \log M$.
\end{lemma}

\begin{proof}
Following again the arguments of Stewart and Tijdeman \cite{Stewart:1997} we consider the following inequality
 $$ 0\leq \frac ba p^{\alpha_3-\alpha_5}q^{\beta_3-\beta_5}r^{\gamma_3-\gamma_5}-1=\frac ba \cdot \frac{ad+1}{bd+1}-1=\frac{b-a}{a(bd+1)}\leq \frac 1{d}$$
 which implies
 \begin{equation}\label{eq:Lamb3}
 \Lambda_3:=\left|\log \frac ba+A' \log p + B'\log q +C' \log r\right|\leq \frac 2d,
 \end{equation}
 where $A'=\alpha_3-\alpha_5$, $B'=\beta_3-\beta_5$ and $C'=\gamma_3-\gamma_5$.
 We apply Matveev's result with
 \begin{gather*}
 D=1, \quad n=4, \quad \kappa=1,\quad b_1=1, \quad b_2=A',\quad b_3=B',\quad b_4=C', \\
 \eta_1=\frac ab, \quad \eta_1=p, \quad \eta_2=q, \quad \eta_3=r,\\
 A_1=\log \left(\max\{a,b\}\right), \quad A_2=\log p, \quad A_3=\log q, \quad A_4=\log r.
\end{gather*}
Since 
$$b<ab+1=\exp(\alpha_1 \log p+\beta_1 \log q+\gamma_1\log r)\leq \exp(M_0)$$
and $\max\{A',B',C'\}\leq M$ we get
$$
 1.893\cdot 10^{12} M_0 \log p \log q \log r \log M >\log d-\log 2.
$$
On the other hand we have
\begin{align*}
d^2>dc+1&\geq \exp\left(\max\{A\log p, B \log q, C \log r\}\right)\\
& \geq \exp(M\min\{\log p,\log q,\log r\}).
\end{align*}
Combining these two inequalities yields the content of the lemma.
\end{proof}

Now the proof of Proposition \ref{prop:first-upper-bound} is a combination of Lemmas \ref{lem:bound124} and \ref{lem:bound-absolute-easy}. 

\begin{proof}[Proof of Proposition \ref{prop:first-upper-bound}]
By inserting the bound for $M_0$ obtained by Lemma \ref{lem:bound124} into the bound provided by Lemma \ref{lem:bound-absolute-easy} we obtain the following inequality
$$M< 9.03 \cdot 10^{22} (\log P)^5 (\log 2M)^2.$$
An application of Lemma \ref{lem:pdw} yields
$$M< 3.62 \cdot 10^{23} (\log P)^5 \log( 3.62\cdot 10^{23} (\log P)^5)^2.$$

If we put $P=5$ and $P=100$ we obtain the other bounds stated in the proposition.
\end{proof}

Although we have an upper bound for $M$ in Proposition \ref{prop:first-upper-bound} this upper bound provides only rather huge bounds for all the exponents. Moreover, the linear form of logarithms \eqref{eq:Lamb3} is not suitable for applying standard reduction schemes like LLL-reduction. Therefore we give a more detailed account on how to find smaller upper bounds for all the exponents.

First, we will prove that all exponents with one possible exception can be bounded in terms of $M_0$ and $M_4$. In particular, Proposition \ref{prop:one-is-large} will be useful in the bound reduction process which we will describe in Section \ref{sec:reduction}. Before we can state and prove Proposition \ref{prop:one-is-large} we prove another helpful lemma:

\begin{lemma}\label{lem:cases}
  There exists a permutation $\sigma$ of $\{1,\dots,6\}$ such that 
  $$\alpha_{\sigma(1)}\leq \alpha_{\sigma(2)}\leq \alpha_{\sigma(3)}\leq \alpha_{\sigma(4)}\leq \alpha_{\sigma(5)}\leq \alpha_{\sigma(6)}.$$
  and such that one of the following two relations holds:
 \begin{itemize}
  \item $\alpha_{\sigma(1)}=\alpha_{\sigma(2)}=\alpha_{\sigma(3)}\leq\alpha_{\sigma(4)} \leq \alpha_{\sigma(5)}\leq \alpha_{\sigma (6)}$ and no two indices out of $\sigma(1)$, $\sigma(2)$ and $\sigma(3)$ are complementary,
  \item $\alpha_{\sigma(1)}=\alpha_{\sigma(2)}\leq \alpha_{\sigma(3)}=\alpha_{\sigma(4)} \leq \alpha_{\sigma(5)}\leq \alpha_{\sigma (6)}$ and $\sigma(1)$ and $\sigma(2)$ are complementary.
 \end{itemize}
\end{lemma}

\begin{proof}
Let $\sigma$ be any permutation of $\{1,\dots,6\}$ such that 
$$\alpha_{\sigma(1)}\leq \alpha_{\sigma(2)}\leq \alpha_{\sigma(3)}\leq \alpha_{\sigma(4)}\leq \alpha_{\sigma(5)}\leq \alpha_{\sigma(6)}.$$
Obviously such a permutation exists.

Since Lemma \ref{lem:min_alpha} we know that $\alpha_{\sigma(1)}=\alpha_{\sigma(2)}$.
 Assume that $\sigma(1)$ and $\sigma(2)$ are not complementary. Then there is exactly one $S$-unit equation out of the system \eqref{eq:Sunit-sys-gen} such that this unit equation contains the index $\sigma(1)$ but not $\sigma(2)$. Since $\alpha_{\sigma(1)}$ is minimal it is equal to one exponent of this unit equation, i.e. $\alpha_{\sigma(1)}=\alpha_{\sigma(2)}= \alpha_{\sigma(3)}$.
 
 If $\sigma(1)$ and $\sigma(3)$ are complementary then there is exactly one $S$-unit equation out of the system \eqref{eq:Sunit-sys-gen} that does not contain the indices $\sigma(1)$ and $\sigma(3)$ but contains the index $\sigma(2)$. But since $\alpha_{\sigma(2)}$ is minimal it is equal to $\alpha_{\sigma(4)}$ and we have $\alpha_{\sigma(1)}=\alpha_{\sigma(2)}=\alpha_{\sigma(3)}= \alpha_{\sigma(4)}$. By exchanging the values of $\sigma(2)$ and $\sigma(3)$ the numbers $\sigma(1)$ and $\sigma(2)$ are complementary and we have found a permutation that satisfies the second case described in the lemma. A similar argument holds if $\sigma(2)$ and $\sigma(3)$ are complementary. Thus we have proved: If $\sigma(1)$ and $\sigma(2)$ are not complementary, then either the first case holds or the second case holds after rearranging the order of the $\alpha$'s, i.e. we have found a suitable permutation $\sigma$.
 
 Therefore let us assume that $\sigma(1)$ and $\sigma(2)$ are complementary. Then there is a unique $S$-unit equation out of the system \eqref{eq:Sunit-sys-gen} such that this unit equation does not contain the indices $\sigma(1)$ and $\sigma(2)$. Thus by Lemma \ref{lem:min_alpha} we obtain that $\alpha_{\sigma(3)}= \alpha_{\sigma(4)}$ and we are in the second case of the lemma.
\end{proof}

Similar as in the case of the exponents of $p$ we let $\tau$ and $\rho$ be permutations of $\{1,\dots,6\}$ such that
$$\beta_{\tau(1)}\leq \beta_{\tau(2)}\leq \beta_{\tau(3)}\leq \beta_{\tau(4)}\leq \beta_{\tau(5)}\leq \beta_{\tau(6)} $$
and
$$ \gamma_{\rho(1)}\leq \gamma_{\rho(2)}\leq \gamma_{\rho(3)}\leq \gamma_{\rho(4)}\leq \gamma_{\rho(5)}\leq \gamma_{\rho(6)}$$
respectively. By exchanging the roles of the $\alpha$'s, $\beta$'s and $\gamma$'s we can show that permutations $\tau$ and $\rho$ with analogous properties as stated in Lemma \ref{lem:cases} exist. Let us fix the permutations $\sigma, \rho$ and $\tau$. We are now in a position to state the next proposition that will prove to be useful in reducing the huge upper bounds we get from Proposition \ref{prop:first-upper-bound}.

\begin{proposition}\label{prop:one-is-large}
 Assume that $\log s_1,\log s_2 \leq M_0$ and that $\log s_4\leq M_4$. Moreover, let
 \begin{gather*}
 B_p=\log p \max_{\substack{|\beta| \leq M_4/\log q \\ |\gamma| \leq M_4/\log r}}\left\{\nu_p(q^\beta r^\gamma-1)\right\}, \qquad
 B_q=\log q \max_{\substack{|\alpha| \leq M_4/\log p \\ |\gamma| \leq M_4/\log r}}\left\{ \nu_q(p^\alpha r^\gamma-1)\right\},\\
 B_r=\log r \max_{\substack{|\alpha| \leq M_4/\log p \\ |\beta| \leq M_4/\log q}}\left\{\nu_r(p^\alpha q^\beta-1)\right\}
 \end{gather*}
 and $\mathcal B=\max\{B_p,B_q,B_r\}$.
 Then we have
  \begin{gather*}
   \alpha_{\sigma(5)}\leq \frac{\max\{M_0+B_p,M_4\}}{\log p},\qquad \beta_{\tau(5)}\leq \frac{\max\{M_0+B_q,M_4\}}{\log q},\\
   \gamma_{\rho(5)}\leq \frac{\max\{M_0+B_r,M_4\}}{\log r}.
  \end{gather*}
  In particular, we have 
$$ \max\{\alpha_{\sigma(5)} \log p ,\beta_{\tau(5)} \log q ,\gamma_{\rho(5)} \log r\}  \leq \max\{M_0+\mathcal B,M_4\}:=M_5.$$
\end{proposition}

\begin{proof}
 We give only the details for the proof of the upper bound for $\alpha_{\sigma(5)}$ since the upper bounds for $\beta_{\tau(5)}$ and $\gamma_{\rho(5)}$ can be deduced by the same argument.

First, let us assume that we are in the first case of Lemma \ref{lem:cases}. Let us assume for the moment that $|\{1,2,4\}\cap \{\sigma(1),\sigma(2),\sigma(3)\}|\leq 1$. In this case we conclude that $|\{1,2,4\}\cap \{\sigma(4),\sigma(5),\sigma(6)\}|\geq 2$ and therefore $\sigma(5)\in \{1,2,4\}$ or $\sigma(6)\in \{1,2,4\}$. Thus in any case $\alpha_{\sigma(5)} \leq M_4/\log p$.
 
 Now, let us assume that we are again in the first case of Lemma \ref{lem:cases} and that $|\{1,2,4\}\cap \{\sigma(1),\sigma(2),\sigma(3)\}|= 2$. We may assume that $\sigma(3) \not\in \{1,2,4\}$. If either $\sigma(5)$ or $\sigma(6)$ are contained in $\{1,2,4\}$, then we easily conclude that $\alpha_{\sigma(5)} \leq M_4/\log p$. Thus we may assume that $\sigma(4)\in \{1,2,4\}$. Further, we deduce that $\sigma(3)$ and $\sigma(4)$ are complementary, and without loss of generality we may assume that $\sigma(1)$ and $\sigma(6)$ as well as $\sigma(2)$ and $\sigma(5)$ are complementary. Let us consider the unit equation
 $$s_{\sigma(6)}s_{\sigma(1)} - s_{\sigma(6)} -s_{\sigma(1)}=s_{\sigma(5)}s_{\sigma(2)} - s_{\sigma(5)} -s_{\sigma(2)}. $$
 Dividing through $p^{\alpha_{\sigma(1)}}$ and collecting the terms not divisible by $p$ on one side of the equation we obtain
 $$q^{\beta_{\sigma(2)}}r^{\gamma_{\sigma(2)}}-q^{\beta_{\sigma(1)}}r^{\gamma_{\sigma(1)}}=
 p^{\alpha_{\sigma(5)}-\alpha_{\sigma(1)}} \times \mathcal N$$
 where $\mathcal N$ is an integer. Thus we deduce that 
 $$\frac{\mathcal B}{\log p}\geq \frac{B_p}{\log p}\geq \nu_p\left(q^{\beta_{\sigma(2)}-\beta_{\sigma(1)}}r^{\gamma_{\sigma(2)}-\gamma_{\sigma(1)}}-1 \right)\geq \alpha_{\sigma(5)}-\alpha_{\sigma(1)}\geq \alpha_{\sigma(5)}- \frac{M_0}{\log p} $$
 which implies the upper bound for $\alpha_{\sigma(5)}$. Note that by our choice that $\sigma(1),\sigma(2)\in \{1,2,4\}$ the exponents satisfy
 $|\beta_{\sigma(2)}-\beta_{\sigma(1)}|\leq \beta_4$ and $|\gamma_{\sigma(2)}-\gamma_{\sigma(1)}|\leq \gamma_4$.
 
 Assume again that we are in the first case of Lemma \ref{lem:cases} and that $\{1,2,4\}= \{\sigma(1),\sigma(2),\sigma(3)\}$. In this case we can choose $\sigma$ such that $\sigma(1)$ and $\sigma(6)$ as well as $\sigma(2)$ and $\sigma(5)$ are complementary. Now the same line of arguments as in the previous paragraph yields the same upper bound for  $\alpha_{\sigma(5)}$.
 
 Now, we assume that we are in the second case of Lemma \ref{lem:cases}. Since $\sigma(1)$ and $\sigma(2)$ are complementary exactly one of the two is an element of $\{1,2,4\}$. Assume now that $|\{\sigma(3),\sigma(4)\}\cap\{1,2,4\}|\leq 1$, then either $\sigma(5)$ or $\sigma(6)$ is contained in $\{1,2,4\}$ and we deduce in any case that $\alpha_{\sigma(5)}\leq M_4/\log p$. Therefore we may assume that $\sigma(3),\sigma(4)\in \{1,2,4\}$. However $\sigma(1)$ and $\sigma(2)$ are complementary and since $\sigma(3)$ and $\sigma(4)$ are not we can choose $\sigma$ such that $\sigma(3)$ and $\sigma(5)$ as well as $\sigma(4)$ and $\sigma(6)$ are complementary. By the same reasoning as in the proof of the first case we deduce that
 $$\frac{\mathcal B}{\log p}\geq \frac{B_p}{\log p}\geq \nu_p\left(q^{\beta_{\sigma(4)}-\beta_{\sigma(3)}}r^{\gamma_{\sigma(4)}-\gamma_{\sigma(3)}}-1 \right)\geq \alpha_{\sigma(5)}-\alpha_{\sigma(3)}.$$
 Since we are in the second case of Lemma \ref{lem:cases} we know that $\alpha_{\sigma(1)}$ and $\alpha_{\sigma(2)}$ are complementary and therefore 
 $$\{\sigma(3),\sigma(4),\sigma(5),\sigma(6)\} \cap \{1,2\}\neq \emptyset$$
 which implies that $\alpha_{\sigma(3)}\leq M_0/\log p$ and we obtain the upper bound for $\alpha_{\sigma(5)}$ also in the second case of the proposition.
\end{proof}

 The next lemma shows that if $M$ is large, then the indices of the large exponents are of a special form.
 
 \begin{lemma}\label{lem:diff-big}
  Let us assume that  $\alpha_{\sigma(5)}\leq M_5/\log p,\beta_{\tau(5)}\leq M_5/\log q$ and $\gamma_{\rho(5)}\leq M_5/\log r$. Then either $M\leq \frac{M_0+3M_5}{\log(\min\{p,q,r\})}$ or
  $\{\sigma(6),\tau(6),\rho(6)\}=\{3,5,6\}$.
 \end{lemma}
 
 \begin{proof}
  Assume that $\{\sigma(6),\tau(6),\rho(6)\}\neq \{3,5,6\}$ and let $*\in\{3,5,6\}$ be such that $*\not\in \{\sigma(6),\tau(6),\rho(6)\}$. We obviously have that 
  $$d<s_*\leq p^{\alpha_{\sigma(5)}}q^{\beta_{\tau(5)}}r^{\gamma_{\rho(5)}}\leq \exp(3M_5)$$
  and also
  $$c<s_2=p^{\alpha_2}q^{\beta_2}r^{\gamma_2}\leq \exp(M_0).$$
  Therefore we obtain 
  $$\max\{p^{\alpha_{\sigma(6)}},q^{\beta_{\tau(6)}},r^{\gamma_{\rho(6)}}\}\leq s_6=cd+1\leq \exp(M_0+3M_5).$$
  Now, taking logarithms yields the content of the lemma.
 \end{proof}
 
 Let us denote by $\N=\{0,1,2,\dots\}$ the set of the non negative integers. Now, we can provide upper bounds for the largest exponents.

\begin{proposition}\label{lem:bound-absolute}
 Let us assume that $M> \frac{M_0+3M_5}{\log(\min\{p,q,r\})}$, $\sigma(6)$ and $x$ are complementary indices and that $y,z$ is another pair of complementary indices, with $y\in\{1,2,4\}$. If $\alpha_y=\alpha_z<\alpha_x$ (Case I) we put
 \begin{multline*}
 \mathcal M_{x,y,z}=\left\{ (a_x,b_x,c_x,a_y,b_y,c_y,b_z,c_z)\in \N^8\: :\phantom{\frac{M_4}{\log q}}\right.\\
  0\leq a_x\log p+b_x\log q+c_x\log r,a_y\log p+b_y\log q+c_y\log r \leq M_4, \; \text{and}\\
  a_y<a_x\; \text{and}\;\left.\left(0\leq b_z\leq \frac{M_5}{\log q},0\leq c_z\leq M \; \text{or}\;\;  0\leq b_z\leq M ,0\leq c_z\leq \frac{M_5}{\log r} \right) \right\},
 \end{multline*}
 if $\alpha_x=\alpha_z<\alpha_y$ (Case II) we put
 \begin{multline*}
 \mathcal M_{x,y,z}=\left\{ (a_x,b_x,c_x,a_y,b_y,c_y,b_z,c_z)\in \N^8\: :\phantom{\frac{M_4}{\log q}}\right.\\
  0\leq a_x\log p+b_x\log q+c_x\log r,a_y\log p+b_y\log q+c_y\log r \leq M_4, \; \text{and}\\
  a_x<a_y\; \text{and}\;\left.\left(0\leq b_z\leq \frac{M_5}{\log q},0\leq c_z\leq M \; \text{or}\;\;  0\leq b_z\leq M ,0\leq c_z\leq \frac{M_5}{\log r} \right) \right\},
 \end{multline*}
 and if $\alpha_x=\alpha_y$ (Case III) we put 
 \begin{multline*}
 \mathcal M_{x,y,z}=\left\{ (a_x,b_x,c_x,a_y,b_y,c_y,b_z,c_z)\in \N^8\: :\phantom{\frac{M_4}{\log q}}\right.\\
 0\leq a_x\log p+b_x\log q+c_x\log r,a_y\log p+b_y\log q+c_y\log r \leq M_4, \; \text{and}\\
  a_x=a_y\; \text{and}\;\left.\left(0\leq b_z\leq \frac{M_5}{\log q},0\leq c_z\leq M \; \text{or}\;\;  0\leq b_z\leq M ,0\leq c_z\leq \frac{M_5}{\log r} \right) \right\}.
 \end{multline*}
We define
\begin{align*}
 \mathcal C_p&=\max_{\mathcal M_{x,y,z}}\left\{\nu_p\left( \frac{q^{b_z}r^{c_z}\left(p^{a_y}q^{b_y}r^{c_y}-1\right)}{q^{b_y}r^{c_y}-p^{a_x-a_y}q^{b_x}r^{c_x}}-1\right) \right\}& &\text{Case I},\\
 \mathcal C_p&=\max_{\mathcal M_{x,y,z}}\left\{\nu_p\left( \frac{q^{b_z}r^{c_z}\left(p^{a_y}q^{b_y}r^{c_y}-1\right)}{p^{a_y-a_x}q^{b_y}r^{c_y}-q^{b_x}r^{c_x}}-1\right) \right\}& &\text{Case II},\\
 \mathcal C_p&=\max_{\mathcal M_{x,y,z}}\left\{\nu_p\left( \frac{q^{b_z}r^{c_z}\left(p^{a_y}q^{b_y}r^{c_y}-1\right)_p}{ \left(q^{b_y}r^{c_y}-q^{b_x}r^{c_x}\right)_p}-1 \right) \right\}& &\text{Case III},\\
 \end{align*}
 where $(\cdot)_p$ denotes the $p$-free part. Then 
 $$\alpha_{\sigma(6)}\leq \left\{ \begin{array}{cl} \frac{M_4}{\log p}+\mathcal C_p & \text{Cases I and II;}\\ \\ \frac{M_4+B_p}{\log p}+\mathcal C_p & \text{Case III.}\end{array} \right. $$
\end{proposition}

\begin{proof}
 Let us assume that Case I holds. Then we consider the $S$-unit equation
 $$s_zs_y-s_z-s_y=s_{\sigma(6)}s_x -s_{\sigma(6)}-s_x$$
 which turns into
 $$\frac{s_z(s_y-1)}{s_y-s_x} -1=\frac{s_{\sigma(6)} (s_x-1)}{s_y-s_x}.$$
 First, we note that $\nu_p(s_y-s_x)=\alpha_y=\alpha_z$ and since $\alpha_x>0$ we have
 \begin{equation}\label{eq:bound-a6}
 \nu_p\left(\frac{s_{\sigma(6)} (s_x-1)}{s_y-s_x}\right)=\alpha_{\sigma(6)}-\alpha_y.
 \end{equation}
 Note that with these notations and assumptions we have after canceling a common factor $p^{\alpha_y}=p^{\alpha_z}$ the equation
 $$\frac{s_z(s_y-1)}{s_y-s_x} -1=\frac{q^{\beta_z}r^{\gamma_z}\left(p^{\alpha_y}q^{\beta_y}r^{\gamma_y}-1\right)}{q^{\beta_y}r^{\gamma_y}-p^{\alpha_x}q^{\beta_x}r^{\gamma_x}}-1$$
 with non negative exponents $\alpha_x,\alpha_y,\beta_x,\beta_y,\gamma_x$ and $\gamma_y$ satisfying $\alpha_y<\alpha_x$ and
 $$0\leq \alpha_x\log p+\beta_x\log q+\gamma_x\log r,\alpha_y\log p+\beta_y\log q+\gamma_y\log r \leq M_4.$$
 Moreover due to Lemma \ref{lem:diff-big} we have either $0\leq \beta_z\leq \frac{M_5}{\log q}$ and $0\leq \gamma_z\leq M$ or $0\leq \beta_z\leq M$ and $0\leq \gamma_z\leq \frac{M_5}{\log r}$. That is $(\alpha_x,\beta_x,\gamma_x,\alpha_y,\beta_y,\gamma_y,\beta_z,\gamma_z)\in \mathcal M_{x,y,z}$.
  Thus we conclude that
  $$\mathcal C_p\geq \nu_p \left( \frac{s_z(s_y-1)}{s_y-s_x} -1\right)=\nu_p\left(\frac{s_{\sigma(6)} (s_x-1)}{s_y-s_x}\right)=\alpha_{\sigma(6)}-\alpha_y.$$
  Which implies in view of $\alpha_y\leq M_4/\log p$ and \eqref{eq:bound-a6} the upper bound for $\alpha_{\sigma(6)}$.
  
  Almost the same arguments apply in the case that $\alpha_x=\alpha_z<\alpha_y$ holds, i.e. that Case II holds. We only have to exchanging the roles of $x$ and $y$. 
  
  In Case III, that is we have $\alpha_x=\alpha_y$, the proof runs along similar arguments. That is we consider the equation
  $$\frac{s_z(s_y-1)}{s_y-s_x} -1=\frac{s_{\sigma(6)} (s_x-1)}{s_y-s_x}. $$
  However this time we note that 
  $$\nu_p(s_y-s_x)= \alpha_x+\nu_p(q^{\beta_y-\beta_x}r^{\gamma_y-\gamma_x}-1) \leq \frac{M_4+B_p}{\log p}.$$
  This implies that 
  $$\alpha_{\sigma(6)}\leq \nu_p\left(\frac{s_z(s_y-1)}{s_y-s_x} -1\right) +\nu_p(s_y-s_x),$$
  and therefore the upper bound for $\alpha_{\sigma(6)}$. Finally let us note that a non vanishing factor $p$ of $\frac{s_z(s_y-1)}{s_y-s_x}$ would imply that $\nu_p\left(\frac{s_z(s_y-1)}{s_y-s_x} -1\right)\leq 0$, thus we may eliminate all possible powers of $p$ in the formula for $\mathcal C_p$, which accounts in the usage of $p$-free parts in the formulas for Case II.
\end{proof}

Let us note that we can prove similar results for $\beta_{\tau(6)}$ and $\gamma_{\rho(6)}$ by exchanging the roles of $p,q$ and $r$ respectively. For the sake of completeness we state the results for $\beta_{\tau(6)}$ and $\gamma_{\rho(6)}$ in the appendix but do not provide a proof since the proofs are identical after an appropriate permutation of $p,q$ and $r$. 

If we have no explicit bounds for $M_4$ and $M_5$ this proposition is hard to apply. However, if we have small, explicit upper bounds for $M_4$ and $M_5$ we can estimate $\mathcal C_p$ by applying lower bounds for linear forms in two $p$-adic logarithms (for details see Section \ref{sec:reduction}) which will yield reasonable small bound for $M$.

\section{The LLL-reduction}\label{sec:LLL}

In this section we gather some basic facts on LLL-reduced bases, approximation lattices and their applications to Diophantine problems as they can be found in \cite[Chapters 5 and 6]{Smart:DiGL}.

Let $\LL\subseteq \R^k$ be a $k$-dimensional lattice with LLL-reduced basis
$b_1,\dots,b_k$ and denote by $B$ be the matrix with columns $b_1,\dots, b_k$. Moreover, we denote by $b^*_1,\dots,b^*_k$ the orthogonal basis of $\R^k$ which we obtain
by applying the Gram-Schmidt process to the basis $b_1,\dots,b_k$. In particular, we have that
$$b^*_i=b_i-\sum_{j=1}^{i-1}\mu_{i,j}b^*_j, \qquad \mu_{i,j}=\frac{\langle b_i,b_j\rangle}{\langle b_j^*,b_j^*\rangle}.$$
Further, let us define
\begin{equation*}
l(\LL,y)=\begin{cases}
\min_{x \in \LL} \|x-y\|, & y \not\in \LL \\
\min_{0 \neq x \in \LL} \|y\|, & y \in \LL,
\end{cases}
\end{equation*}
where $\|\cdot\|$ denotes the euclidean norm on $\R^k$. It is well known,
that by applying the LLL-algorithm it is possible to give in polynomial time
a lower bound for $l(\LL,y) \geq \tilde c_1$ (see e.g. \cite[Section 5.4]{Smart:DiGL}):

\begin{lemma}\label{lem:lattice}
 Let $y\in \R^k$, $z=B^{-1}y$. Furthermore we define
 \begin{itemize}
 \item If $y\not\in\LL$ let $i_0$ be the largest index such that $z_{i_0}\neq 0$ and put $\sigma=\{z_{i_0}\}$, where $\{\cdot\}$ denotes the distance to the nearest integer.
 \item If $y\in \LL$ we put $\sigma=1$.
 \end{itemize}
 Finally let
 $$\tilde c_2=\max_{1\leq j\leq k}\left\{\frac{\|b_1\|^2}{\|b_j^*\|^2} \right\}.$$
 Then we have
 $$l(\LL,y)^2\geq \tilde c_2^{-1}\sigma^2 \|b_1\|^2=\tilde c_1.$$
\end{lemma}

In our applications we suppose that we are given real numbers $\eta_0,\eta_1,\dots,\eta_k$  linearly independent over $\Q$ and two positive constants
$\tilde c_3,\tilde c_4$ such that
\begin{equation} \label{eq:redform1}
|\eta_0+x_1\eta_1+\dots+x_k\eta_k| \le \tilde c_3\exp(-\tilde c_4H),
\end{equation}
where the integers $x_i$ are bounded by $|x_i| \leq X_i$ with $X_i$ given upper bounds for $1 \leq i \leq k$.
We write $X_0=\max_{1 \le i \le s}\{X_i\}$. The basic idea in such a situation, due to de Weger \cite{deWeger:1987}, is to approximate the linear form
\eqref{eq:redform1} by an approximation lattice. Namely, we consider the lattice $\LL$ generated by the columns of the matrix
\begin{equation*}
\mathcal{A}=\begin{pmatrix}
    1 & 0 & \dots & 0  & 0 \\
    0 & 1 & \dots & 0  & 0 \\
    \vdots & \vdots & \vdots & \vdots & \vdots \\
    0 & 0 & \dots & 1  & 0 \\
    \lfloor{\tilde C\eta_1}\rfloor & \lfloor{\tilde C\eta_2}\rfloor & \dots & \lfloor{\tilde C\eta_{k-1}}\rfloor & \lfloor{\tilde C\eta_k}\rfloor
\end{pmatrix}
\end{equation*}
where $\tilde C$ is a large constant usually of the size of about $X_0^k$. Let us assume that we have an LLL-reduced basis $b_1,\dots,b_k$ of $\LL$ and that we have a lower bound
$l(\LL,y)>\tilde c_1$ with $y=(0,0,\dots ,-\lfloor{C\eta_0}\rfloor)$. Note that $\tilde c_1$ can be computed by using the results of Lemma \ref{lem:lattice}. Then we have with these notations the following lemma (e.g. see \cite[Lemma VI.1]{Smart:DiGL}):

\begin{lemma}\label{lem:real-reduce}
Assume that $S=\sum_{i=1}^{k-1}X_i^2$ and $T=\frac{1+\sum_{i=1}^k{X_i}}{2}$. If $\tilde c_1^2 \ge T^2+S$, then inequality \eqref{eq:redform1} implies that
we have either $x_1=x_2=\dots=x_{k-1}=0$ and $x_k=-\frac{\lfloor{\tilde C\eta_0}\rfloor)}{\lfloor{\tilde C\eta_k}\rfloor)}$ or
\begin{equation} \label{eq:reduction-real}
H \leq \frac{1}{\tilde c_4}\left(\log(\tilde C\tilde c_3)-\log\left(\sqrt{\tilde c_1^2-S}-T\right) \right).
\end{equation}
\end{lemma}

\section{Reduction of the bounds}\label{sec:reduction}

In this section we describe how we can reduce the rather huge bounds for the exponents $\alpha_i,\beta_i,\gamma_i$ for $i=1,\dots,6$ for a given set of primes $S=\{p,q,r\}$ and how it is in practice possible to enumerate for the given set $S$ all $S$-Diophantine quadruples.
We give the full details for the proof of Theorem \ref{th:main}, i.e. the case that $S=\{2,3,5\}$ and will outline the strategy how to find all $S$-Diophantine quadruples for a general set of three primes $S$. The reduction process follows in eight steps:

\textbf{Step I:} We compute an upper bound for $M$ by using Proposition \ref{prop:first-upper-bound}. In the case that $S=\{2,3,5\}$ we obtain $M<1.26\cdot 10^{28}$. Note that in the case that $P<100$ which is relevant for the proof of Theorem \ref{th:search} we obtain the bound $M<2.88\cdot 10^{30}$.

\textbf{Step II:} We apply the LLL-reduction method explained in Section \ref{sec:LLL} to the linear forms of logarithms \eqref{eq:Lamb1} and \eqref{eq:Lamb2} in order to obtain a small bound for $M_0$.

In particular, we apply Lemma \ref{lem:real-reduce} to the linear form \eqref{eq:Lamb1}. In the case that $S=\{2,3,5\}$ we put $k=3$ and $\eta_1=\log 5, \eta_2=\log 3$ and $\eta_3=\log 2$. Moreover we have $0\leq A',B',C'\leq 2M$, i.e. we have to put $X_1=X_2=X_3=2.52\cdot 10^{28}$. Moreover let us put $\tilde c_3=2$ and $\tilde c_4=1$. With $\tilde C=10^{88}$ the assumptions of Lemma \ref{lem:real-reduce} are satisfied and we obtain
$$2\exp(- 136.4)>|A'\log 2+B'\log 3+C'\log 5|>\frac{2}{p^{\alpha_1}q^{\beta_1}r^{\gamma_1}}$$
which implies $\log s_1 \leq 136.4$. Since the exact same computations also apply to
\eqref{eq:Lamb2} we also obtain that $\log s_2<136.4$ and therefore we have $M_0<136.4$ we conclude that $s_4<s_1s_2<\exp(272.8)$. Thus we have $M_4\leq 272.8$.

\textbf{Step III:}
By a direct computation we are able to compute $B_p$, $B_q$ and $B_r$ by a brute force algorithm. The search implemented in Sage \cite{sage} takes on a usual PC a few seconds. In
particular we obtain in the case that $S=\{2,3,5\}$ the bounds
$$\frac{B_p}{\log p}=18, \qquad \frac{B_q}{\log q}=13, \qquad \frac{B_r}{\log r}=8. $$
By Proposition \ref{prop:one-is-large} this yields
$$
\alpha_{\sigma(5)}\leq 392, \qquad \beta_{\tau(5)}\leq 247, \qquad \gamma_{\rho(5)}\leq 169.
$$
In particular, we have $M_5=M_4\leq 272.8 $.

\textbf{Step IV:} In this step we use the theorem of Bugeaud and Laurent (Lemma~\ref{lem:Bugeaud}) to estimate the quantities $\mathcal C_p$ and obtain a new upper bound for $\alpha_{\sigma(6)}$ due to Proposition~\ref{lem:bound-absolute} (and the upper bounds for $\mathcal C_q$ and $\mathcal C_r$ by Propositions~\ref{lem:bound-absolute-b} and \ref{lem:bound-absolute-c} in the Appendix). Since the computations for $\mathcal C_q$ and $\mathcal C_r$ are similar we give the details only for $\mathcal C_p$ for Case I note that the upper bounds for Cases II and III can be deduced similarly.

To find an upper bound for $\mathcal C_p$ we split up the set $\mathcal M_{x,y,z}$ into two subsets  $\mathcal M'_{x,y,z},\mathcal M''_{x,y,z}\subseteq \mathcal M_{x,y,z}$ with
\begin{multline*}
 \mathcal M'_{x,y,z}=\left\{ (a_x,b_x,c_x,a_y,b_y,c_y,b_z,c_z)\in \N^8\: :\phantom{\frac{M_4}{\log q}}\right.\\
  0\leq a_x\log p+b_x\log q+c_x\log r,a_y\log p+b_y\log q+c_y\log r \leq M_4, \; \text{and}\\
 a_y<a_x\; \text{and}\; \left.0\leq b_z\leq \frac{M_5}{\log q},0\leq c_z\leq M\right\}
 \end{multline*}
and
\begin{multline*}
 \mathcal M''_{x,y,z}=\left\{ (a_x,b_x,c_x,a_y,b_y,c_y,b_z,c_z)\in \N^8\: :\phantom{\frac{M_4}{\log q}}\right.\\
  0\leq a_x\log p+b_x\log q+c_x\log r,a_y\log p+b_y\log q+c_y\log r \leq M_4, \; \text{and}\\
  a_y<a_x\; \text{and}\; \left.0\leq b_z\leq M,0\leq c_z\leq \frac{M_5}{\log r}\right\}
 \end{multline*}
and also consider the two quantities
$$\mathcal C'_p=\max_{\mathcal M'_{x,y,z}}\left\{\nu_p\left( \frac{q^{b_z}r^{c_z}\left(p^{a_y}q^{b_y}r^{c_y}-1\right)}{q^{b_y}r^{c_y}-p^{a_x-a_y}q^{b_x}r^{c_x}}-1\right) \right\}  $$
and
$$\mathcal C''_p=\max_{\mathcal M''_{x,y,z}}\left\{\nu_p\left( \frac{q^{b_z}r^{c_z}\left(p^{a_y}q^{b_y}r^{c_y}-1\right)}{q^{b_y}r^{c_y}-p^{a_x-a_y}q^{b_x}r^{c_x}}-1\right) \right\}.$$
Of course we have $\mathcal C_p=\max\{\mathcal C'_p,\mathcal C''_p\}$ and we have to find
upper bounds for $\mathcal C'_p$ (see Step IV (a)) and $\mathcal C''_p$ (see Step IV (b)). Similarly we can define $\mathcal C'_q, \mathcal C''_q,\mathcal C'_r$ and $\mathcal C''_r$ which will yield bounds for $\mathcal C_q$ and $\mathcal C_r$ respectively.

\textbf{Step IV (a):} In this step we compute an upper bound for $\mathcal C'_p$. We put 
$$\eta_1=r,\qquad 
\eta_2=\frac{q^{b_z}(p^{a_y}q^{b_y}r^{c_y}-1)}{q^{b_y}r^{c_y}-p^{a_x-a_y}q^{b_x}r^{c_x}}$$
and $b_1=c_z$ and $b_2=1$. We find that
\begin{align*}
h(\eta_2) &\leq \max\{b_z\log q +\log (p^{a_y}q^{b_y}r^{c_y}-1),\log (q^{b_y}r^{c_y}-p^{a_x-a_y}q^{b_x}r^{c_x})\}\\ &\leq M_5+M_4
\end{align*}
since we have $0\leq b_z\leq \frac{M_5}{\log q}$ and 
$$0\leq a_x\log p+b_x\log q+c_x\log r,a_y\log p+b_y\log q+c_y\log r \leq M_4.$$
Applying Lemma~\ref{lem:Bugeaud} with this data we obtain a new hopefully smaller upper bound for $A=\alpha_{\sigma(6)}$. 

\textbf{Step IV (b):} In this step we compute an upper bound for $\mathcal C''_p$. In this case we put $\eta_1=q$, 
$$\eta_2=\frac{r^{c_z}(p^{a_y}q^{b_y}r^{c_y}-1)}{q^{b_y}r^{c_y}-p^{a_x-a_y}q^{b_x}r^{c_x}}$$
and $b_1=b_z$ and $b_2=1$. We find that
$$h(\eta_2)\leq \max\{c_z\log r +\log (s_y-1),\log (s_y-s_x)\}\leq M_5+M_4$$
since we have $0\leq c_z\leq \frac{M_5}{\log r}$ and 
$$0\leq a_x\log p+b_x\log q+c_x\log r,a_y\log p+b_y\log q+c_y\log r \leq M_4.$$
Again we apply Lemma~\ref{lem:Bugeaud} with this data and we obtain an upper bound for $A=\alpha_{\sigma(6)}$ in this case. 

Similarly we find upper bounds in Cases II and III. We also find upper bounds for $B=\beta_{\tau(6)}$ and $C=\gamma_{\rho(6)}$ by using Propositions \ref{lem:bound-absolute-b} and \ref{lem:bound-absolute-c} instead of Propositions \ref{lem:bound-absolute}.

In the case that $S=\{2,3,5\}$ we obtain 
$$A<7.17 \cdot 10^8,\qquad B<1.73\cdot 10^8, \qquad C<6.33\cdot 10^7.$$

\textbf{Step V:} We repeat the process of the Steps II-IV iteratively with the new found bounds for $M_0$ and $M$ four more times and will find significantly smaller bounds. In the case that $S=\{2,3,5\}$ we obtain after five LLL-reductions the upper bounds $M_0<33.624$, $M_4=M_5<67.248$ and
$$A<4.51 \cdot 10^6,\qquad B< 1.3 \cdot 10^6,\qquad  C< 1.01 \cdot 10^6.$$

We implemented the Steps I-V in Sage. It took a usual desk top PC only a few seconds to obtain the bounds stated in Step V. Before we proceed with Step VI let us summarize
what we proved so far in the case that $S=\{2,3,5\}$:

\begin{proposition}
 Assume that $(a,b,c,d)$ is a $\{2,3,5\}$-Diophantine Quadruple with $a<b<c<d$. Then $\log s_1,\log s_2 <33.624$ and $\log s_4<67.25$.
\end{proposition}

\textbf{Step VI:}  
In this step we find all $\{p,q,r\}$-Diophantine triples $(a,b,c)$ such that $\log(ab+1)=\log s_1,\log(ac+1)=\log s_2<M_0$ and $\log (bc+1)=\log s_4<M_4$. Therefore we proceed as follows:

\begin{enumerate}
 \item We enumerate all $S$-units $s_4=p^{\alpha_4}q^{\beta_4}r^{\gamma_4}$ with $\log s_4<M_4$.
 \item We enumerate all $S$-units $s_2=p^{\alpha_2}q^{\beta_2}r^{\gamma_2}$ with $\log s_2<\min\{\log s_4, M_0\}$ and $s_2\nmid s_4$. In particular, we consider only those $s_2$ which satisfy $\gamma_2>\gamma_4$ in case that $\alpha_2\leq \alpha_4$ and $\beta_2\leq \beta_4$.
 \item We compute for all pairs $s_2,s_4$ the quantity $B_c=\frac{s_4-s_2}{\gcd (s_2,s_4)}$. Due to Lemma~\ref{lem:abcd-div} we have $c\leq B_c$. If $(B_c-1)B_c+1<s_4$ we discard the pair $s_2,s_4$ since otherwise we would have
 $$bc+1<(B_c-1)B_c+1<s_4=bc+1,$$
 a contradiction.
 \item We enumerate all $S$-units $s_1=p^{\alpha_1}q^{\beta_1}r^{\gamma_1}$ with $s_1<s_2$ such that
 \begin{itemize}
 \item if $p=2$ and $\alpha_4,\alpha_2>1$, then $\alpha_1\leq 1$;
  \item if $p\equiv 3\mod 4$ and $\alpha_4,\alpha_2>0$, then $\alpha_1=0$;
  \item if $q\equiv 3\mod 4$ and $\beta_4,\beta_2>0$, then $\beta_1=0$;
  \item if $r\equiv 3\mod 4$ and $\gamma_4,\gamma_2>0$, then $\gamma_1=0$.
 \end{itemize}
That is we enumerate all $S$-units $s_1$ such that the content of Lemma \ref{lem:quad-res} is not violated for the $S$-Diophantine triple $(a,b,c)$.
\item For all these triples $(s_1,s_2,s_4)$ we compute $a_\square=\frac{(s_1-1)(s_2-1)}{s_4-1}$. We discard all triples $(s_1,s_2,s_4)$ for which $a_\square$ is not the square of an integer. For the remaining triples we compute $a=\sqrt{a_\square}$.
\item We discard all triples $(s_1,s_2,s_4)$ for which $b=\frac{s_1-1}a$ and $c=\frac{s_2-1}a$ are not integers.
\item We store all triples $(a,b,c)$ in a list \texttt{Triples}.
\end{enumerate}

In the case that $S=\{2,3,5\}$ the list \texttt{Triples} consists of the $S$-Diophantine triples
\begin{gather*}
 \{(1, 3, 5), (1, 7, 9), (1, 15, 17), (1, 23, 89), (1, 5, 19), (1, 2, 4), (1, 5, 7), (1, 47, 49),\\
 (1, 19, 485), (1, 7, 23), (1, 17, 19), (1, 31, 47), (1, 49, 119), (1, 7, 3749), (1, 3, 53),\\
 (1, 11, 29), (1, 9, 71), (1, 2, 7), (1, 7, 17), (1, 4, 11), (1, 11, 49), (3, 13, 83), (1, 17, 127),\\
 (2, 7, 1562), (1, 19, 1151), (1, 3, 8), (1, 9, 11), (1, 17, 47), (1, 159, 161), (1, 29, 31),\\
 (1, 19, 71), (1, 127, 287), (1, 7, 107), (1, 89, 8191), (1, 24, 26), (1, 44, 71), (1, 29, 431)\}. 
\end{gather*}

Let us note that Step VI is the bottle neck of our algorithm. It took more than six and a half hours to find all these triples using a usual PC (Intel i7-8700) on a single core.

\textbf{Step VII:} Let us fix one $S$-Diophantine triple $(a,b,c)$ from the list \texttt{Triples} computed in Step VI. Let us assume that this triple can be extended to a $S$-Diophantine quadruple. In this step we reduce the previously found huge upper bound for $s_3=ad+1$ by using the LLL-reduction (Lemma \ref{lem:real-reduce}). Therefore we consider the inequality
$$ \frac{b}{a}\cdot \frac{ad+1}{bd+1}-1=\frac{b-a}{abd+a}<\frac{1}{ad+1}$$
and taking logarithms we obtain
\begin{equation}\label{eq:trip_red}
 \left|\log (b/a)+(\alpha_3-\alpha_5)\log p+(\beta_3-\beta_5)\log q+(\gamma_3-\gamma_5)\log r\right| <\frac 2{s_3}
\end{equation}
We apply Lemma \ref{lem:real-reduce} to this Diophantine inequality that is we have 
$$k=3,\quad \eta_1=\log r,\quad  \eta_2=\log q,\quad \eta_4=\log p, \quad \eta_0=\log (b/a).$$
Moreover we put $X_1=A$,  $X_2=B$ and $X_3=C$, where $A,B$ and $C$ are the bounds for $\alpha_{\sigma(6)}, \beta_{\tau(6)}$ and $\gamma_{\rho(6)}$ found in Step V.
We distinguish now between two cases.

\textbf{Case I:} Let us assume that $p,q,r$ and $a/b$ are not multiplicatively dependent. For example this happens for the $\{2,3,5\}$-Diophantine triple $(1,7,9)$. We compute $l(\mathcal L,y)$ by using Lemma \ref{lem:lattice} with
$$
\mathcal{A}=\begin{pmatrix}
    1 & 0 & 0 \\
    0 & 1 & 0 \\
    \lfloor{C\log r}\rfloor & \lfloor{C\log q}\rfloor & \lfloor{C\log p}\rfloor
\end{pmatrix}=\begin{pmatrix}
    1 & 0 & 0 \\
    0 & 1 & 0 \\
    \lfloor{\tilde C\log 5}\rfloor & \lfloor{\tilde C\log 3}\rfloor & \lfloor{\tilde C\log 2}\rfloor
\end{pmatrix}
$$
and 
$$y=(0,0,-\lfloor - \tilde C \log (b/a) \rfloor)^T=(0,0,-\lfloor - \tilde C \log 7 \rfloor)^T.$$
With $\tilde C=10^{25}$ we obtain by Lemma \ref{lem:real-reduce} the bound $\log s_3<46.6$ in the case of the $\{2,3,5\}$-Diophantine triple $(1,7,9)$.

\textbf{Case II:}  Let us assume that $p,q,r$ and $a/b$ are multiplicatively dependent and that $b/a=p^{x_p}q^{x_q}r^{x_r}$. Then Inequality \eqref{eq:trip_red} turns into
$$\left|(\alpha_3-\alpha_5+x_p)\log p+(\beta_3-\beta_5+x_q)\log q+(\gamma_3-\gamma_5+x_r)\log r\right| <\frac 2{s_3}.$$
Thus we compute $l(\mathcal L,y)$ by using Lemma \ref{lem:lattice} with
$$
\mathcal{A}=\begin{pmatrix}
    1 & 0 & 0 \\
    0 & 1 & 0 \\
    \lfloor{C\log r}\rfloor & \lfloor{C\log q}\rfloor & \lfloor{C\log p}\rfloor
\end{pmatrix}
$$
and $y=(0,0,0)^T$. We put $X_1=A+|x_p|$,  $X_2= B+|x_q|$ and $X_3= C+|x_r|$ and $\tilde C$ sufficiently large and obtain an upper bound for $\log s_3$. For example in the case of the $\{2,3,5\}$-Diophantine triple $(1,3,5)$ we have $x_2=x_5=0$ and $x_3=1$ and obtain with this strategy $\log s_3<39.4$ if we choose $\tilde C=10^{25}$.

Let us note that in the case that $S=\{2,3,5\}$ we obtain in all cases the upper bound $\log s_3<40$.

\textbf{Step VIII:}
Let $M_3$ denote the upper bound for $\log s_3$ found in Step VII. We enumerate all $ad+1=s_3=p^{\alpha_3}q^{\beta_3}r^{\gamma_3}$ with $\log s_3<M_3$ and for each triple $(a,b,c)$ found in Step VI we compute $d=\frac{s_3-1}a$. We discard all quadruples $(a,b,c,d)$ for which $d$ is not an integer and for which $bd+1$ and $cd+1$ are not $S$-units.

Note that in the case that $S=\{2,3,5\}$ we have $M_3=40$. However a rather quick computer search yields that no triple can be extended to a quadruple and the proof of Theorem \ref{th:main} is therefore complete.

\begin{remark}
Finally let us note that the algorithm presented in this paper would also work for arbitrary $S$ with $|S|\geq 3$ after some modifications. For instance in Step~IV we would have to apply the results of Yu on linear forms in $p$-adic logarithms (e.g. the results from \cite{Yu:2007}) instead of the results due to Bugeaud and Laurent \cite{Bugeaud:1996} on linear forms in two $p$-adic logarithms. This will lead to much larger bounds for $M_0$ and $M_5$ than what we get in the case $|S|=3$.

Moreover, Step VI in which we search for extendable triples seems to get unfeasible in the case $|S|\geq 4$. Thus all together a systematic search for $S$-Diophantine quadruples or even quintuples in the case $|S|\geq 4$ seems to be unfeasible with the method presented in this paper.
\end{remark}

\section{Appendix -- Bounds for $\beta_{\tau(6)}$ and $\gamma_{\rho(6)}$}

In this appendix we explicitly state the upper bounds for $\beta_{\tau(6)}$ and $\gamma_{\rho(6)}$ we get by adjusting the proof of Proposition \ref{lem:bound-absolute}. A bound for $\beta_{\tau(6)}$ is given by:

\begin{proposition}\label{lem:bound-absolute-b}
 Let us assume that $M> \frac{M_0+3M_5}{\log(\min\{p,q,r\})}$, $\tau(6)$ and $x$ are complementary indices and that $y,z$ is another pair of complementary indices, with $y\in\{1,2,4\}$. If $\beta_y=\beta_z<\beta_x$ (Case I) we put
 \begin{multline*}
 \mathcal M_{x,y,z}=\left\{ (a_x,b_x,c_x,a_y,b_y,c_y,a_z,c_z)\in \N^8\: :\phantom{\frac{M_4}{\log q}}\right.\\
  0\leq a_x\log p+b_x\log q+c_x\log r,a_y\log p+b_y\log q+c_y\log r \leq M_4, \; \text{and}\\
  b_y<b_x\; \text{and}\;\left.\left(0\leq a_z\leq \frac{M_5}{\log p},0\leq c_z\leq M \; \text{or}\;\;  0\leq a_z\leq M ,0\leq c_z\leq \frac{M_5}{\log r} \right) \right\},
 \end{multline*}
 if $\beta_x=\beta_z<\beta_y$ (Case II) we put
 \begin{multline*}
 \mathcal M_{x,y,z}=\left\{ (a_x,b_x,c_x,a_y,b_y,c_y,a_z,c_z)\in \N^8\: :\phantom{\frac{M_4}{\log q}}\right.\\
  0\leq a_x\log p+b_x\log q+c_x\log r,a_y\log p+b_y\log q+c_y\log r \leq M_4, \; \text{and}\\
  b_x<b_y\; \text{and}\;\left.\left(0\leq a_z\leq \frac{M_5}{\log p},0\leq c_z\leq M \; \text{or}\;\;  0\leq a_z\leq M ,0\leq c_z\leq \frac{M_5}{\log r} \right) \right\},
 \end{multline*}
 and if $\beta_x=\beta_y$ (Case III) we put 
 \begin{multline*}
 \mathcal M_{x,y,z}=\left\{ (a_x,b_x,c_x,a_y,b_y,c_y,a_z,c_z)\in \N^8\: :\phantom{\frac{M_4}{\log q}}\right.\\
 0\leq a_x\log p+b_x\log q+c_x\log r,a_y\log p+b_y\log q+c_y\log r \leq M_4, \; \text{and}\\
  b_x=b_y\;\text{and}\;\left.\left(0\leq a_z\leq \frac{M_5}{\log p},0\leq c_z\leq M \; \text{or}\;\;  0\leq a_z\leq M ,0\leq c_z\leq \frac{M_5}{\log r} \right) \right\}.
 \end{multline*}
We define
\begin{align*}
 \mathcal C_q&=\max_{\mathcal M_{x,y,z}}\left\{\nu_q\left( \frac{p^{a_z}r^{c_z}\left(p^{a_y}q^{b_y}r^{c_y}-1\right)}{p^{a_y}r^{c_y}-p^{a_x}q^{b_x-b_y}r^{c_x}}-1\right) \right\}& &\text{Case I},\\
 \mathcal C_q&=\max_{\mathcal M_{x,y,z}}\left\{\nu_q\left( \frac{p^{a_z}r^{c_z}\left(p^{a_y}q^{b_y}r^{c_y}-1\right)}{p^{a_y}q^{b_y-b_x}r^{c_y}-p^{a_x}r^{c_x}}-1\right) \right\}& &\text{Case II},\\
 \mathcal C_q&=\max_{\mathcal M_{x,y,z}}\left\{\nu_q\left( \frac{p^{a_z}r^{c_z}\left(p^{a_y}q^{b_y}r^{c_y}-1\right)_q}{ \left(p^{a_y}r^{c_y}-p^{a_x}r^{c_x}\right)_q}-1 \right) \right\}& &\text{Case III},\\
 \end{align*}
 where $(\cdot)_q$ denotes the $q$-free part. Then 
 $$\beta_{\tau(6)}\leq \left\{ \begin{array}{cl} \frac{M_4}{\log q}+\mathcal C_q & \text{Cases I and II;}\\ \\ \frac{M_4+B_q}{\log q}+\mathcal C_q & \text{Case III.}\end{array} \right. $$
\end{proposition}

A bound for $\gamma_{\rho(6)}$ is given by:

\begin{proposition}\label{lem:bound-absolute-c}
 Let us assume that $M> \frac{M_0+3M_5}{\log(\min\{p,q,r\})}$, $\rho(6)$ and $x$ are complementary indices and that $y,z$ is another pair of complementary indices, with $y\in\{1,2,4\}$. If $\gamma_y=\gamma_z<\gamma_x$ (Case I) we put
 \begin{multline*}
 \mathcal M_{x,y,z}=\left\{ (a_x,b_x,c_x,a_y,b_y,c_y,a_z,b_z)\in \N^8\: :\phantom{\frac{M_4}{\log q}}\right.\\
  0\leq a_x\log p+b_x\log q+c_x\log r,a_y\log p+b_y\log q+c_y\log r \leq M_4, \; \text{and}\\
  c_y<c_x\; \text{and}\;\left.\left(0\leq a_z\leq \frac{M_5}{\log p},0\leq b_z\leq M \; \text{or}\;\;  0\leq a_z\leq M ,0\leq b_z\leq \frac{M_5}{\log q} \right) \right\},
 \end{multline*}
 if $\alpha_x=\alpha_z<\alpha_y$ (Case II) we put
 \begin{multline*}
 \mathcal M_{x,y,z}=\left\{ (a_x,b_x,c_x,a_y,b_y,c_y,a_z,b_z)\in \N^8\: :\phantom{\frac{M_4}{\log q}}\right.\\
  0\leq a_x\log p+b_x\log q+c_x\log r,a_y\log p+b_y\log q+c_y\log r \leq M_4, \; \text{and}\\
  c_x<c_y\; \text{and}\;\left.\left(0\leq a_z\leq \frac{M_5}{\log p},0\leq c_z\leq M \; \text{or}\;\;  0\leq a_z\leq M ,0\leq c_z\leq \frac{M_5}{\log r} \right) \right\},
 \end{multline*}
 and if $\alpha_x=\alpha_y$ (Case III) we put 
 \begin{multline*}
 \mathcal M_{x,y,z}=\left\{ (a_x,b_x,c_x,a_y,b_y,c_y,a_z,b_z)\in \N^8\: :\phantom{\frac{M_4}{\log q}}\right.\\
 0\leq a_x\log p+b_x\log q+c_x\log r,a_y\log p+b_y\log q+c_y\log r \leq M_4, \; \text{and}\\
  c_x=c_y\; \text{and}\;\left.\left(0\leq a_z\leq \frac{M_5}{\log p},0\leq b_z\leq M \; \text{or}\;\;  0\leq a_z\leq M ,0\leq b_z\leq \frac{M_5}{\log q} \right) \right\}.
 \end{multline*}
We define
\begin{align*}
 \mathcal C_r&=\max_{\mathcal M_{x,y,z}}\left\{\nu_r\left( \frac{p^{a_z}q^{b_z}\left(p^{a_y}q^{b_y}r^{c_y}-1\right)}{p^{a_y}q^{b_y}-p^{a_x}q^{b_x}r^{c_x-c_y}}-1\right) \right\}& &\text{Case I},\\
 \mathcal C_r&=\max_{\mathcal M_{x,y,z}}\left\{\nu_r\left( \frac{p^{a_z}q^{b_z}\left(p^{a_y}q^{b_y}r^{c_y}-1\right)}{p^{a_y}q^{b_y}r^{c_y-c_x}-p^{a_x}q^{b_x}}-1\right) \right\}& &\text{Case II},\\
 \mathcal C_r&=\max_{\mathcal M_{x,y,z}}\left\{\nu_r\left( \frac{p^{a_z}q^{b_z}\left(p^{a_y}q^{b_y}r^{c_y}-1\right)_r}{ \left(p^{a_y}q^{b_y}-p^{a_x}q^{b_x}\right)_r}-1 \right) \right\}& &\text{Case III},\\
 \end{align*}
 where $(\cdot)_r$ denotes the $r$-free part. Then 
 $$\gamma_{\rho(6)}\leq \left\{ \begin{array}{cl} \frac{M_4}{\log r}+\mathcal C_r & \text{Cases I and II;}\\ \\ \frac{M_4+B_r}{\log r}+\mathcal C_r & \text{Case III.}\end{array} \right. $$
\end{proposition}

\def\cprime{$'$}


\end{document}